\documentclass[11pt,article]{amsart}
\usepackage{amssymb}
\usepackage{amsfonts}
\usepackage{amsbsy}
\usepackage{latexsym}
\usepackage{enumitem}
\theoremstyle{plain}

 \theoremstyle{remark} 

\newcommand{\Z}{{\mathbb Z}}

\newcommand{\inner}[1]{\langle #1 \rangle}

\newtheorem {theo} {\bf Theorem} [section]

\newtheorem {coro} [theo] {\bf Corollary}
\newtheorem {lem} [theo] {\bf Lemma}
\newtheorem {defi} {\bf Definition}[section]

\newcommand{\R}{\mathbb R\,}

\newcommand{\ip}[2]{\langle\,#1\, , \, #2 \, \rangle}

\numberwithin{equation}{section}
\begin{document}
\title{A duality principle for groups  II: Multi-frames meet super-frames}

\author{Randu Balan}\address{Department of Mathematics\\
University of Maryland\\ College Park, MD 20742}
\email{rvbalan@math.umd.edu}
\author{Dorin Dutkay}\address{Department of Mathematics\\
University of Central Florida\\ Orlando, FL 32816}
\email{ddutkay@mail.ucf.edu}
\author{Deguang Han}
\address{Department of Mathematics\\
University of Central Florida\\ Orlando, FL 32816}
\email{deguang.han@ucf.edu}
\author{David Larson}\address{Department of Mathematics\\
Texas A\&M University\\ College Station, TX}
\email{larson@math.tamu.edu}
\author{Franz Luef}\address{Department of Mathematics\\
Norwegian University of Science and Technology\\ Trondheim, Noway}
\email{Franz.luef@ntnu.no}
\thanks{$(*)$  Dorin Dutkay is  partially supported by a grant from the Simons Foundation (No. 228539). Deguang Han is partially supported by NSF DMS-1403400 and DMS-1712602}
\date{\today}

\begin{abstract}  The duality principle for group representations developed in \cite{DHL-JFA, HL_BLM} exhibits a fact that the well-known duality principle in Gabor analysis is not an isolated incident but a  more general phenomenon residing in the context of group representation theory.  There are two other well-known fundamental properties in Gabor analysis: The Wexler-Raz biorthogonality and the Fundamental Identity of Gabor analysis. In this paper we will show that these fundamental properties remain to be true for general projective unitary group representations. The main purpose of this paper  is present a more general duality theorem which shows that that muti-frame  generators meet super-frame generators through a dual commutant pairs. In particular, for the  Gabor representations $\pi_{\Lambda}$ and $\pi_{\Lambda^{o}}$ with respect to a pair of  dual time-frequency lattices $\Lambda$ and $\Lambda^{o}$ in $\R^{d}\times \R^{d}$ we have  that $\{\pi_{\Lambda}(m, n)g_{1} \oplus ... \oplus \pi_{\Lambda}(m, n)g_{k}\}_{m, n \in \Z^{d}}$ is a frame for $L^{2}(\R^{d})\oplus ... \oplus L^{2}(\R^{d})$  if and only if $\cup_{i=1}^{k}\{\pi_{\Lambda^{o}}(m, n)g_{i}\}_{m, n\in\Z^{d}}$ is a Riesz sequence, and  $\cup_{i=1}^{k}\{\pi_{\Lambda}(m, n)g_{i}\}_{m, n\in\Z^{d}}$ is a frame for $L^{2}(\R^{d})$ if and only if  $\{\pi_{\Lambda^{o}}(m, n)g_{1} \oplus ... \oplus \pi_{\Lambda^{o}}(m, n)g_{k}\}_{m, n \in \Z^{d}}$ is a Riesz sequence. This appears to be new even in the context of Gabor analysis.
\end{abstract}
\maketitle

\section{Introduction }
\setcounter{figure}{0} \setcounter{equation}{0}

In this paper we continue the investigation on the duality phenomenon for projective unitary group representations. The purpose of this paper is two-fold: First we  prove that the Wexler-Raz biorthogonality and the Fundamental Identity in Gabor analysis also reflect a general phenomenon  for more  general projective unitary representations of any countable groups. Secondly we  establish a new duality principle connecting the muti-frame generators and super-frame generators, which is  new even in the context of Gabor analysis.

Recall that a {\it frame} for a Hilbert space $H$ is a sequence
 $\{x_{n}\}_{n\in I}$ in $H$ with the property that there
exist positive constants $A, B > 0$ such that
\begin{equation} \label{eg1.1}
A\|x\|^{2}\leq \sum_{n\in I}|\ip{x}{x_{n}}|^{2} \leq B \|x\|^{2}
\end{equation}
holds for every $x\in H$. {\it A tight frame} refers to the case
when $A = B$, and a {\it Parseval frame}  refers to the case when
$A = B = 1$. In the case that (\ref{eg1.1}) hold only for all
$x\in \overline{span}\{x_{n}\}$, then we say that $\{x_{n}\}$ is a
{\it frame sequence}, i.e., it is a frame for its closed linear
span. If we only require the right-hand side of the inequality
(\ref{eg1.1}), then $\{x_{n}\}$ is called a {\it Bessel sequence}. Similarly,  a Riesz sequence is a sequence that is a Riesz basis for its closed linear span. 

Given a sequence $\{x_{n}\}_{n\in I}$ in a Hilbert space $H$. The {\it analysis operator} $\Theta: H \rightarrow \ell^{2}(I)$ is defined by
$$
\Theta(x) = \sum_{n\in I}\ip{x}{x_{n}} e_{n}, \ \ \  x\in H,
$$
where $\{e_{n}\}_{n\in I}$ is the standard orthonormal basis for $\ell^{2}(I)$ and the domain of $\Theta$ is the set of all $x\in H$ such that $\{\ip{x}{x_{n}}\}_{n\in I} \in \ell^{2}(I)$. Clearly the domain of $\Theta$ is $H$ if $\{x_{n}\}_{n\in I}$ is a frame sequence and the range of $\Theta$ is $\ell^{2}(I)$ if $\{x_{n}\}_{n\in I}$ is a Riesz sequence.

One of the well studied classes of frames is the time-frequency lattice representation frames.    Let $\Lambda= A(\Z^{d}\times \Z^{d})$ be a full-rank time-frequency lattices, where $A$ is a $2d\times 2d$ invertible real matrix.  The dual lattice of $\Lambda$ is the time-frequency lattice defined by 
 $$
 \Lambda^{o} = \{\lambda^{o}\in \R^{d}\times \R^{d}: \langle \lambda, \lambda^{o}\rangle \in \Z, \forall \lambda\in\Lambda\}.
 $$
 
 A {\it Gabor (or Weyl-Heisenberg ) family} is
a collection of functions in $L^2(\R^d)$
$$
   {\bf G}(g, \Lambda) = \{e^{2\pi i<\ell, x>} g(x-\kappa): \lambda = (\ell, \kappa)\in\Lambda\}, 
$$
where $g\in L^{2}(\R^{d})$ is the generator of the Gabor family.  A Gabor frame (with a single generator) is a frame of the form  ${\bf G}(g, \Lambda)$. 
Let $E_{\ell}$ and $T_{\kappa}$ be the modulation and
translation unitary operators defined by $ E_{\ell}f(x) = e^{2\pi
i<\ell, x>}f(x) $ and
$
T_{\kappa}f(x) = f( x - \kappa)
$
for all $f\in L^{2}(\R ^{d})$. Then we have  $ {\bf G}(g, \Lambda) =
\{E_{\ell}T_{\kappa}g: \lambda =(\ell, \kappa)\in \Lambda\}$. Hence a Gabor frame is a frame induced by the Gabor representation $\pi_{\Lambda}$ of  the abelian group $\Z^{d}\times
\Z^{d}$ defined by $\pi_{\Lambda}(m, n) \rightarrow
E_{\ell}T_{\kappa}$, where $(\ell, \kappa) = A(m, n)$.

In Gabor analysis, there are several fundamental theorems: Probably the most well-known  one is the {\bf Density Theorem} which tells us that a Gabor frame exists if and only if the $vol(\Lambda) \leq 1$, i.e.,  the density of $\Lambda$ is greater than or equal to one (c.f. \cite{BB, HW1, He07, Rie}), where the density of $\Lambda$ is ${1\over vol(\Lambda)}$ and  $vol(\Lambda)$ is the Lebesgue measure of a fundamental domain of $\Lambda$, which is equal to $|det(A)|$ if $\Lambda = A(\Z^{d}\times \Z^{d})$.

The other well-known properties include the  (Ron-Shen)  duality principle, the Wexler-Raz biorthogonality and the Fundamental Identity of Gabor frames. The duality principle for Gabor frames was independently and
essentially simultaneous discovered by Daubechies, H. Landau, and
Z. Landau \cite{DLL}, Janssen \cite{Jan}, and Ron and Shen
\cite{RS}, and the techniques used in these three articles to
prove the duality principle are quite different. We refer to
\cite{He07} for more details about this principle and its
important applications.

\begin{theo} Let $\Lambda = A\Z^{2d}$ be a lattice with $|det(A)|\leq 1$, and $\Lambda^{0}$ be its dual lattice. Then we have 

(i)   [The duality principle] A Gabor family $\{\pi_{\Lambda}(m, n)g\}$ is a frame (resp. Parserval frame) for $L^{2}(\R^{d})$ if and only if $\{\pi_{\Lambda^{(0)}}(m, n)g\}$ is a Riesz sequence (resp. orthogonal sequence).

(ii)  [The Wexler-Raz biorthogonality]  If  $\{\pi_{\Lambda}(m, n)g\}$ is a frame  for $L^{2}(\R^{d})$, then  $$\langle \pi_{\Lambda^{0}}(m, n)g, S^{-1}g\rangle =  |det A| \delta_{(m, n), (0, 0)},$$  where $S$ is the frame operator for  $\{\pi_{\Lambda}(m, n)g\}$

(iii) [The Fundamental Identity of Gabor Analysis -Janssen representation]  If $f, g, h, k$ are Bessel vectors for $\pi_{\Lambda}$, then 
$$
 \sum_{m, n} \langle f, \pi_{\Lambda}(m, n)g\rangle \langle \pi_{\Lambda}(m, n)h, k\rangle = 
 vol(\Lambda)^{-1}\sum_{m, n} \langle f, \pi_{\Lambda^{\circ}}(m, n)k\rangle \langle  \pi_{\Lambda^{\circ}}(m, n)h, g\rangle
$$

i.e. $$\langle \Theta_{\pi_{\Lambda}, g}(f), \Theta_{\pi_{\Lambda}, h}(k)\rangle  = vol(\Lambda)^{-1} \langle \Theta_{\pi_{\Lambda^{\circ}}, k}(f), \Theta_{\pi_{\Lambda^{\circ}}, h}(g)\rangle ,  $$
where $\Theta_{\pi_{\Lambda}, g}$ (similarly for $\Theta_{\pi_{\Lambda^{\circ}}, h}$ etc.) is the analysis operator for $\{\pi_{\Lambda}(m, n)g: m, n\in\Z^{d}\}$.
\end{theo} 

Considering the fact all basic properties above establish the intrinsic connections between the Gabor representations $\pi_{\Lambda}$ and $\pi_{\Lambda^{0}}$,  which can be viewed as projective unitary representations of the abelian group $\Z^{d}\times \Z^{d}$, it is not surprising to expect that this might be a phenomenon for general projective unitary representations on any countable groups. We investigated the duality principle for general groups in \cite{HL_BLM} and \cite{DHL-JFA}, and established the following theorem:

\begin{theo}\label{thm-main1}   Let $\pi$ be a frame representation and  $(\pi, \sigma)$ be
a dual commutant  pair (see Definition  \ref{dual-pair}) of projective unitary representations of $G$ on a Hilbert space $H$. Then $\{\pi(g)\xi\}_{g\in G}$ is a frame (respectively,
a tight frame) for $H$ if and only if $\{\sigma(g)\xi\}_{g\in
G}$ is a Riesz sequence (respectively, an orthogonal sequence).
\end{theo}

Oner of the central problems concerning the duality principle is the existence problem of dual commutant pairs $(\pi, \sigma)$ for a group $G$ and/or  for a given  representations $\pi$. This turns out to be a very difficulty problem due to the following result \cite{DHL-JFA}:

\begin{theo}\label{thm-ICC} {\it Let $\pi = \lambda|_{P}$ be a subrepresentation of the left
regular representation $\lambda$  of an ICC (infinite conjugacy class) group
$G$ and $P$ is an orthogonal projection in the commutant  $\lambda(G)'$ of $\lambda(G)$. Then the following are
equivalent:

(i)  $\lambda(G)'$ and $P\lambda(G)'P$ are isomorphic von Neumann
algebras;

(ii) there exists a group representation $\sigma$ such that $(\pi,
\sigma)$ form a dual pair. }
\end{theo}

For the free groups $\mathcal{F}_{n}$ with $n$-generators ($n\geq 2$), it is a longstanding problem whether all  their group  von Neumann algebras are $*$-isomorphic.  It was well-known  \cite{Dykema, Rad-Inv} that either all the von Neumann algebras
$P\lambda(\mathcal{F}_{n})'P$ ( $0\neq
P\in\lambda(\mathcal{F}_{n})'$) are  $*$-isomorphic, or no two of
them are $*$-isomorphic. This implies that the classification problem is also equivalent to the question  whether there exists a proper projection $P\in \lambda(\mathcal{F}_{n})'$ such that $\lambda(\mathcal{F}_{n})'$ and $P\lambda(\mathcal{F}_{n})'P$ are isomorphic von Neumann algebras.The above Theorem \ref{thm-ICC} shows that the  existence problem of commutant dual pairs for free groups  is also equivalent to the longstanding classification problem for free group von Neumann algebras.  

There are plenty of examples where we know that the dual commutant pairs exist. For example,  if $G$ is either an abelian group or an amenable ICC (infinite conjugate class) group,
then for every projection $0\neq P\in \lambda(G)'$, there exists a
unitary representation $\sigma$ of $G$ such that $(\lambda|_{P},
\sigma)$ is a  commutant dual pair, where $\lambda|_{P}$ is the subrepresentation of the left regular representation $\lambda$ restricted to $range(P)$. On the other hand, there exists an ICC group (e.g., $G =\Z^{2}\rtimes SL(2, \Z)$),
such that none of the nontrivial subrepresentations $\lambda|_{P}$
admits a commutant dual pair (c.f. \cite{Conn, DHL-JFA, Dykema, Popa_PNSA, Popa_AnnMath, Rad-JAMS, Rad-Inv}). These examples demonstrate the complexity of the existence problem which remains open for almost all the cases.

In this paper we  first prove that the Wexler-Raz biorthogonality and the Fundamental Identity in Gabor analysis remain to be true  for more  general projective unitary representations of any countable group $G$. Secondly we shall establish the duality principle connecting the muti-frame generators and super-frame generators, which is  new even in the context of Gabor analysis.   In order to state our main results we need to recall  some necessary definitions, notations and terminologies related to frames and frame representations.


 Recall (cf. \cite{Va})
that a {\it projective unitary representation} $\pi$ for a
countable  group $G$ is a mapping $g\rightarrow \pi(g)$ from $G$
into the group $U(H)$ of all the unitary operators on a separable
Hilbert space $H$ such that $\pi(g)\pi(h) = \mu(g, h)\pi(gh)$ for
all $g, h\in G$, where $\mu(g, h)$ is a scalar-valued function on
$G\times G$ taking values in the circle group $\mathbb{T}$. This
function $\mu(g, h)$ is then called a {\it multiplier or
$2$-cocycle} of $\pi$. In this case we also say that $\pi$ is a
$\mu$-projective unitary representation.  It is clear from the
definition that we have
\begin{enumerate}[label=(\roman*)]

\item $\mu(g_{1}, g_{2}g_{3}) \mu(g_{2}, g_{3}) = \mu(g_{1}g_{2},
g_{3}) \mu(g_{1}, g_{2})$ for all $g_{1}, g_{2}, g_{3}\in G$,

\item $\mu(g, e) = \mu(e, g) = 1$ for all $g\in G$, where $e$
denotes the group unit of $G$.
\end{enumerate}

Any function $\mu: G\times G \rightarrow \mathbb{T}$ satisfying
$(i)$ -- $(ii)$ above will be called a {\it multiplier} for $G$.
It follows from $(i)$ and $(ii)$ that we also have

(iii) $\mu(g, g^{-1}) = \mu(g^{-1}, g)$ holds for all $g\in G$.

Similar to the group unitary representation case,   the left and
right regular projective representations with a prescribed
multiplier $\mu$ for $G$  can be defined by
$$
\lambda_{g}\chi_{h} = \mu(g, h)\chi_{gh}, \ \ \  h\in G,
$$
and
$$
\rho_{g}\chi_{h} = \mu(h, g^{-1})\chi_{hg^{-1}}, \ \ \  h\in G,
$$
where $\{\chi_{g}: g\in G\}$ is the standard orthonormal basis for
$\ell^{2}(G)$. Clearly, $\lambda_{g}$ and $r_{g}$ are unitary
operators on $\ell^{2}(G)$. Moreover,  $\lambda$ is a
$\mu$-projective unitary representation of $G$ with multiplier
$\mu(g, h)$ and $\rho$ is a projective unitary representation of
$G$ with multiplier $\overline{\mu(g, h)}$. The representations
$\lambda$ and $\rho$ are called the {\it left regular
$\mu$-projective representation} and the {\it right regular
$\mu$-projective representation}, respectively, of $G$. 

Given a projective unitary representation $\pi$ of a countable
group $G$ on a Hilbert space $H$, a vector $\xi \in H$ is called a
{\it complete frame vector (resp. complete tight frame vector,
complete Parseval frame vector)} for $\pi$ if $ \{\pi(g)\xi\}_{
g\in G}$ (here we view this as a sequence indexed by $G$) is a
frame (resp. tight frame, Parseval frame) for the whole Hilbert
space $H$, and is just called a {\it frame sequence vector (resp. tight
frame  sequence vector,  Parseval  sequence frame vector)} for $\pi$ if $
\{\pi(g)\xi\}_{ g\in G}$  is a {\it frame sequence (resp. tight
frame sequence, Parseval frame sequence)}.  Riesz sequence vector and Bessel vector  can be defined similarly.  We will use $\mathcal{B}_{\pi}$ to denote the set
of all the Bessel vectors of $\pi$.

For Gabor representations, both $\pi_{\Lambda}$ and $\pi_{\Lambda^{\circ}}$  are projective unitary representation of the
group $\Z^{d}\times \Z^{d}$. Moreover, it is well-known  that one of
two projective unitary representations $\pi_{\Lambda}$ and
$\pi_{\Lambda^{o}}$ for the group $G = \Z^{d}\times\Z^{d}$ must be
a frame representation and the other admits a Riesz vector. So we
can always assume that $\pi_{\Lambda}$ is a frame representation
of $\Z^{d}\times \Z^{d}$ and hence $\pi_{\Lambda^{o}}$ admits a
Riesz vector.  Moreover, we also have $\pi_{\Lambda}(G)' =
\pi_{\Lambda^{\circ}}(G)''$, and both representations share the same
Bessel vectors, where $\pi_{\Lambda}(G)' $ is the commutant of $\pi(G)$. This leads to the following definition

\begin{defi} \cite{DHL-JFA} \label{dual-pair} Let  $\pi$ and
$\sigma$  be two projective unitary representations of a countable group $G$ on the same Hilbert space $H$. We say that $(\pi, \sigma)$ is
 a {\it commutant pair} if $\pi(G)' = \sigma(G)''$, and a {\it  dual commutant pair} if they satisfy the following two additional conditions:
\begin{enumerate}[label=(\roman*)]
\item  $\mathcal{B}_{\pi} = \mathcal{B}_{\sigma}$.

\item  One of them admits a complete frame generator and the other one admits a Riesz sequence generator.
\end{enumerate}
\end{defi}

For any projective representation $\pi$ of a countable group $G$
on a Hilbert space $H$  and  $x\in H$,  the {\it analysis operator
} $\Theta_{x, \pi}$ (or $\Theta_{x}$ if $\pi$ is well-understood from the context) for $x$ from $\mathcal{D}(\Theta_{x}) (\subseteq
H)$ to $\ell^{2}(G)$ is defined by
$$
\Theta_{x}(y) = \sum_{g\in G}\inner{y, \pi(g)x}\chi_{g},
$$
where $\mathcal{D}(\Theta_{x}) = \{y\in H: \sum_{g\in G}|\inner{y,
\pi(g)x}|^{2} < \infty\}$ is the domain space of $\Theta_{x}$.
Clearly, $\mathcal{B}_{\pi} \subseteq \mathcal{D}(\Theta_{x})$
holds for every $x\in H$. In the case that $\mathcal{B}_{\pi}$ is
dense in $H$, we have that $\Theta_{x}$ is a densely defined and
closable linear operator from $\mathcal{B}_\pi$ to $\ell^2(G)$
(cf. \cite{GH}). Moreover, $x\in \mathcal{B}_{\pi}$ if and only if
$\Theta_{x}$ is a bounded linear operator on $H$, which in turn is
equivalent to the condition that $\mathcal{D}(\Theta_{x}) = H$. It is useful to note that 
$\Theta_{\eta}^{*}\Theta_{\xi}$ commutes with $\pi(G)$ if $\xi, \eta\in\mathcal{B}_{\pi}$. Thus, if
$\xi$ is a complete frame vector for $\pi$, then $\eta : =
S^{-1/2}\xi$ is a complete Parseval frame vector for $\pi$,
where $S= \Theta_{\xi}^{*}\Theta_{\xi}$ and is called the
{\it frame operator} for $\xi$ (or {\it Bessel operator} if $\xi$
is a Bessel vector).

It was proved in \cite{GH1} that for any complete Parseval frame vector $\eta$ for $\pi$, $Tr_{\pi(G)'}(A) = \langle A\eta, \eta\rangle $ defines a faithful normal trace on $\pi(G)'$. In the case of Gabor representation $\pi(\Lambda)$, it turns out that $Tr_{\pi(G)'}(I) = vol(\Lambda)$. Thus the following can be viewed as the  generalized Wexler-Raz biorthogonality and the fundamental frame duality identity for general frame representations.

\begin{theo}\label{main-thm1}Let $\pi$ be a frame representation and  $(\pi, \sigma)$ be
a  dual commutant  pair of projective unitary representations of $G$ on $H$. 

\begin{enumerate}[label=(\roman*)]

\item If  $\{\pi(g)\xi\}$ is a frame  for $H$, then  $$\langle \sigma(\xi), S^{-1}\xi\rangle =  Tr_{\pi(G)'}(I)\delta_{g, e},$$  where $S$ is the frame operator for  $\{\pi(g)\xi\}$,  $e$ is the group unit of $G$ and $Tr_{\pi(G)'}(I) = ||S^{-1/2}\xi||^{2}$.

\item If $\xi, \eta, x, y$ are Bessel vectors for $\pi$, then 
$$
\sum_{g\in G} \langle x, \pi(g)\xi\rangle  \langle \pi(g)\eta, y\rangle ={1\over Tr_{\pi(G)'}(I)} \sum_{g\in G} \sum_{g\in G} \langle x, \sigma(g)(y)\rangle  \langle \sigma(g)\eta, \xi\rangle.$$
i.e. $\langle \Theta_{\xi, \pi}(x), \Theta_{\eta, \pi}(y)\rangle ={1\over Tr_{\pi(G)'}(I)} \langle \Theta_{y, \sigma}(x), \Theta_{\eta, \sigma}(\xi)\rangle.$
\end{enumerate}

\end{theo}

Our second main theorem deals with the duality principle for multi-frame and super-frame generators. 

\begin{defi} Let $\pi$ be projective unitary representation  of a countable
group $G$ on a Hilbert space $H$ and let $\xi_{1}, ... , \xi_{n} \in H$. We say that $\vec{\xi} = (\xi_{1}, ... , \xi_{n})$ is 
\begin{enumerate}[label=(\roman*)]

\item  a multi-frame vector for $\pi$ if $\{\pi(g)\xi_{i}: g\in G, i = 1, ..., n\}$ is a frame for $H$, and 

\item  a super-frame vector if each $\{\pi(g)\xi_{i}: g\in G\}$ is a frame for $H$ and $\Theta_{\xi_{i}}(H) \perp \Theta_{\xi_{j}}(H)$ for $i\neq j$. 
\end{enumerate}
\end{defi}

Parseval multi-frame vector  and Parseval super-frame vector can be defined similarly.  We remark that the concept of super-frames was first introduced  and systematically studied by Balan \cite{Ba1, Ba2},  Han and Larson \cite{HL} in the 1990's, and since then it has generated a host of research activities (c.f. \cite{DHP_JFA, DHPS-MathAnn, DBP, DG07,  GH, GH2, GH3, GH4, Han2_Parseval} and the references therein). 

\begin{theo}\label{main-thm2}  Let $\pi$ be a frame representation and  $(\pi, \sigma)$ be
a commutant  dual pair of projective unitary representations of $G$ on $H$. Let ${\vec \xi} = (\xi_{1}, ... , \xi_{n})$ . Then we have

(i)  $\vec{\xi}$ is a  super-frame vector for $\pi$ if and only if $\{\sigma(g)\xi_{j}: g\in G, j = 1, ..., n\}$ is Riesz sequence in $H$

(ii)  $\vec{\xi}$ is a multi-frame vector for $\pi$ if and only if $\{\sigma(g)\xi_{1}\oplus ... \oplus \sigma(g)\xi_{n}: g\in G\}$ is a Riesz sequence in $H\oplus ... \oplus H$.
\end{theo}

Since the time-frequency representations $\pi_{\Lambda}$ and $\pi_{\Lambda^{o}}$ form a dual commutant pair, we immediately have the following:

\begin{coro} \label{coro-1} Let $\Lambda$ be a time-frequency lattice and $\Lambda^{o}$ be its dual lattice. Let $g_{1}, ... , g_{k}\in L^{2}(\R^{d})$. Then 

(i)  $\{\pi_{\Lambda}(m, n)g_{1} \oplus ... \oplus \pi_{\Lambda}(m, n)g_{k}\}_{m, n \in \Z^{d}}$ is a frame for $L^{2}(\R^{d})\oplus ... \oplus L^{2}(\R^{d})$  if and only if $\cup_{i=1}^{k}\{\pi_{\Lambda^{o}}(m, n)g_{i}\}_{m, n\in\Z^{d}}$ is a Riesz sequence in $L^{2}(\R^{d})$. 

(ii) $\cup_{i=1}^{k}\{\pi_{\Lambda}(m, n)g_{i}\}_{m, n\in\Z^{d}}$ is a frame for $L^{2}(\R^{d})$ if and only if  $\{\pi_{\Lambda^{o}}(m, n)g_{1} \oplus ... \oplus \pi_{\Lambda^{o}}(m, n)g_{k}\}_{m, n \in \Z^{d}}$ is a Riesz sequence $L^{2}(\R^{d})\oplus ... \oplus L^{2}(\R^{d})$.

\end{coro}

\section{ Proof of Theorem \ref{main-thm1}}

We refer to \cite{Di, KR, Mvon} for any standard terminologies and basic properties about von Neumann algebras that will be used in the rest of the paper.
In what follows we will also use $[K]$ to denote the closed subspace generated by a subset $K$ of a Hilbert space $H$. Theorem \ref{thm-main1} and the following  lemmas are needed in the proofs for both Theorem \ref{main-thm1} and Theorem \ref{main-thm2}.

\begin{lem}  \cite{GH1} \label{lem-2.1} Let $\pi$ be a projective
representation of a countable group $G$ on a Hilbert space $H$
such that $\mathcal{B}_{\pi}$ is dense in $H$. Then
$$
\pi(G)' =  \overline{span}^{WOT}\{\Theta_{\eta}^{*}\Theta_{\xi}:
\xi, \eta\in \mathcal{B}_{\pi}\},
$$
where ``$WOT$" denotes the closure in the weak operator topology. 
\end{lem}

\begin{lem} \cite{GH1} \label{lem-2.2} Let $\pi$ be a projective
representation of a countable group $G$ on a Hilbert space $H$
such that $\mathcal{B}_{\pi}$ is dense in $H$. If $\{\pi(g)\xi_{i}, g\in G, i=1, ... , n\}$ is a Parseval frame for $H$, then 
$$
Tr_{\pi(G)'}(A) = \sum_{i=1}^{n}\langle A\xi, \xi\rangle
$$
defines a faithful trace on $\pi(G)'$. Moreover, this is independent the choices of the Parseval multi-frame vector $\vec{\xi} = (\xi_{1}, .. , \xi_{n})$. 
\end{lem}

\begin{lem}\label{prop-sub}   Let $\pi$ be a  projective unitary representation $\pi$ of a countable
group $G$ on a Hilbert space $H$. Then $\pi$ is a frame
representation if and only if $\pi$ is unitarily equivalent to a
subrepresentation of the left regular projective unitary
representation of $G$. Consequently, if  $\pi$ is a frame
representation, then both $\pi(G)'$ and $\pi(G)''$ are finite von
Neumann algebras.
\end{lem}

\begin{lem}\label{param-lem} \cite{GH1, HL} Let $\pi$ be a projective representation of a
countable group $G$ on a Hilbert space $H$ and
$\{\pi(g)\xi\}_{g\in G}$ is a Parseval frame for $H$. Then

\begin{enumerate}[label=(\roman*)]

\item $\{\pi(g)\eta\}_{g\in G}$ is a Parseval frame for $H$ if and
only if there is a unitary operator $U\in \pi(G)''$ such that
$\eta = U\xi$;

\item $\{\pi(g)\eta\}_{g\in G}$ is a frame for $H$ if and only if
there is an invertible operator $U\in \pi(G)''$ such that $\eta =
U\xi$;

\item  $\{\pi(g)\eta\}_{g\in G}$ is a Bessel sequence if and only
if there is an operator $U\in \pi(G)''$ such that $\eta = U\xi$,
i.e., $\mathcal{B}_{\pi} = \pi(G)''\xi$.

\end{enumerate}
\end{lem}

\noindent {\bf Proof of Theorem \ref{main-thm1}.} 

 Let $(\pi, \sigma)$ be a  dual commutant pair of representations for $G$ on a Hilbert space $H$.

(i)  Let $\{\pi(g)\xi\}$ be a frame for $H$ and $S$ be its analysis operator. Write $\eta = S^{-1/2}\xi$. Then $\{\pi(g)\eta\}_{g\in G}$ is a Parserval frame for $H$.  By Lemma \ref{lem-2.2} we have that $Tr_{\pi(G)'}(A) : = \langle A\eta, \eta\rangle $ defines a faithful trace on $w^{*}(\sigma(G))$, where $w^{*}(\sigma(G))$ is the von Neumann algebra generated by $\sigma(G)$ and it is equal to $\pi(G)'$.  Note that $S, \sigma(g) \in \pi(G)'$.  Thus we have
\begin{eqnarray*}
   \langle \sigma(g)\eta, \eta \rangle &=& Tr_{\pi(G)'}(\sigma(g)) = Tr_{\pi(G)'}(S^{-1/2}\sigma(g)S^{1/2}) \\
   & = & \langle S^{-1/2}\sigma(g)S^{1/2}\eta, \eta \rangle = \langle \sigma(g)\xi, S^{-1}\xi \rangle.
\end{eqnarray*}
However,  by Theorem \ref{thm-main1}, $\{\sigma(g)\eta\}_{g\in G}$ is an orthogonal sequence.
Thus we have $\langle \sigma(g)\xi,  S^{-1}\xi\rangle = 0$ for any $g\neq e$. Observe that $\langle \sigma(e)\xi,  S^{-1}\xi\rangle = ||S^{-1/2}\xi||^{2} = Tr_{\pi(G)'}(I)$. So we get the biorthogonality realtion:
$$
\langle \sigma(\xi), S^{-1}\xi\rangle =  Tr_{\pi(G)'}(I)\delta_{g, e}.
$$

(ii) Let $\varphi$  be a Parserval frame vector for $\pi$. Then by Theorem \ref{thm-main1} we get that $\{{1\over \sqrt{Tr_{\pi(G)'}(I)}}\sigma(g)\varphi\}_{g\in G}$ is an orthonormal basis for $[\pi(G)'\varphi]$. Since $\Theta_{\xi, \pi}^{*}\Theta_{\eta, \pi}\in \pi(G)' = w^{*}(\sigma(G))$, we get that 
$\Theta_{\xi, \pi}^{*}\Theta_{\eta, \pi}\varphi \in [\sigma(G)\varphi]$. This implies that 
$$
\Theta_{\xi, \pi}^{*}\Theta_{\eta, \pi}\varphi = (Tr_{\pi(G)'}(I))^{-1}\sum_{g\in G}c_{g}\sigma(g)\varphi,
$$
where $c_{g} = \langle \Theta_{\xi, \pi}^{*}\Theta_{\eta, \pi}\varphi, \sigma(g)\varphi\rangle$. 

By Lemma \ref{param-lem} there is an operator $A\in w^{*}(\pi(G))$ such that $y = A\varphi$. So we have
\begin{eqnarray*}
\Theta_{\xi, \pi}^{*}\Theta_{\eta, \pi}(y) &= & \Theta_{\xi, \pi}^{*}\Theta_{\eta, \pi}(A\varphi) = A\Theta_{\xi, \pi}^{*}\Theta_{\eta, \pi}(\varphi) \\
&=& (Tr_{\pi(G)'}(I))^{-1}\sum_{g\in G}c_{g}A\sigma(g)\varphi\\
&=& (Tr_{\pi(G)'}(I))^{-1}\sum_{g\in G}c_{g}\sigma(g)A\varphi\\
&=& (Tr_{\pi(G)'}(I))^{-1}\sum_{g\in G}c_{g}\sigma(g)y\\
\end{eqnarray*}
Therefore we get 
$$\langle \Theta_{\xi, \pi}(x), \Theta_{\eta, \pi}(y)\rangle = \langle  x, \  \Theta_{\xi, \pi}^{*}\Theta_{\eta, \pi}(y)\rangle = (Tr_{\pi(G)'}(I))^{-1}\sum_{g\in G}\overline{c}_{g}\langle x, \sigma(g)y\rangle.$$

Now we compute $c_{g}$: 
\begin{eqnarray*}
c_{g} &=& \langle \Theta_{\xi, \pi}^{*}\Theta_{\eta, \pi}\varphi, \sigma(g)\varphi\rangle = \langle \Theta_{\eta, \pi}(\varphi), \Theta_{\xi, \pi}(\sigma(g)\varphi)\rangle\\
& = & \sum_{h\in G} \langle \varphi, \pi(h)\eta\rangle \cdot\overline{ \langle \sigma(g)\varphi, \pi(h)\xi \rangle} \\
& = & \sum_{h\in G} \langle \varphi, \pi(h)\eta\rangle \cdot \langle \pi(h)\xi, \sigma(g)\varphi \rangle \\
& = & \sum_{h\in G} \langle \pi(h^{-1})\varphi, \eta\rangle\cdot  \langle \sigma(g^{-1})\xi, \pi(h^{-1})\varphi\rangle\\
& = & \sum_{h\in G} \langle \sigma(g^{-1})\xi, \pi(h^{-1})\varphi\rangle \cdot \langle \pi(h^{-1})\varphi, \eta\rangle\\
& = & \sum_{h\in G} \langle \sigma(g^{-1})\xi, \pi(h)\varphi\rangle \cdot \langle \pi(h)\varphi, \eta\rangle\\
& = &  \langle \sigma(g^{-1})\xi, \eta\rangle = \langle \xi, \sigma(g)\eta\rangle, 
\end{eqnarray*}
where we used the fact that $\sigma(g)$ and $\pi(h)$ commute for all $g, h\in G$, and that $\{\pi(h)\varphi\}_{h\in G}$ is a Parserval frame $H$. 

Finally we have 
\begin{eqnarray*}
\langle \Theta_{\xi, \pi}(x), \Theta_{\eta, \pi}(y)\rangle &= & \langle  x, \  \Theta_{\xi, \pi}^{*}\Theta_{\eta, \pi}(y)\rangle \\
&= & (Tr_{\pi(G)'}(I))^{-1}\sum_{g\in G}\overline{c}_{g}\langle x, \sigma(g)y\rangle \\
& = &(Tr_{\pi(G)'}(I))^{-1}\sum_{g\in G}\langle\sigma(g)\eta, \xi\rangle,\langle x, \sigma(g)y\rangle\\
& = & (Tr_{\pi(G)'}(I))^{-1}\langle \Theta_{y, \sigma}(x), \Theta_{\eta, \sigma}(\xi)\rangle.
\end{eqnarray*}
This completes the proof. \qed

\section{Proof of Theorem \ref{main-thm2}} 

The proof of Theorem \ref{main-thm2} is much more subtle and complicated. While Theorem \ref{thm-main1} will be needed, it is not a direct consequence of the theorem. For the purpose of clarity we divide the proof into two theorems with  one of them concerning the duality for multi-frame generators and the other one dealing with the duality for super-frame generators. We need a series of lemmas for both cases. In what follows we use $H^{(k)}$ to denote the orthogonal direct sum of a Hilbert space $H$ and $\pi^{(k)}$ to denote  direct sum of the representation $\pi$ of $G$ on $H^{(k)}$.  So for any vector $\vec{\xi} = (\xi_{1}, ..., \xi_{k})\in H^{(k)}$, we have $\pi^{(k)}(g)\vec{\xi} = (\pi(g)\xi_{1} , ... , \pi(g)\xi_{k}) = \pi(g)\xi_{1}\oplus ... \oplus \pi(g)\xi_{k}$. We will use the following notations:  Let $\pi$ be a projective unitary representation of $G$ on a Hilbert space $H$.

(i)  For any $\xi\in H$, 
$\Theta_{\xi, \pi}: H \rightarrow \ell^{2}(G)$ is the analysis operator for the sequence $\{\pi(g)\xi\}_{G}$. 

(ii)  For $\xi = (\xi_{1}, ... , \xi_{k})\in H^{(k)}$, $\Theta_{\vec{\xi}, \pi}: H \rightarrow (\ell^{2}(G))^{(k)}$ is the analysis operator for the sequence $\cup_{i=1}^{k}\{\pi(g)\xi_{i}\}_{g\in G}$ defined by
$$
\Theta_{\vec{\xi}, \pi}(x) = \Theta_{\xi_{1}, \pi}(x) \oplus ... \oplus  \Theta_{\xi_{k}, \pi}(x).
$$

(iii) For $\xi = (\xi_{1}, ... , \xi_{k})\in H^{(k)}$, $\Theta_{\vec{\xi}, \pi^{(k)}}: H^{(k)} \rightarrow (\ell^{2}(G))^{(k)}$ is the analysis operator for the sequence $\{\pi^{(k)}(g)\vec{\xi}\}_{g\in G}$.

Clearly $\Theta_{\vec{\xi}, \pi}$ can be viewed as the restriction of $ \Theta_{\vec{\xi}, \pi^{(k)}}$ to the subspace $\{x\oplus ... \oplus x: x\in H\}$ of $H^{(k)}$.

\begin{lem}  \label{lem2.1} Let $\pi$ be a projective unitary representation
of a countable group $G$ on a Hilbert space $H$  such that $\pi(G)'$ is finite.  Assume that $\cup_{i=1}^{k}\{\pi(g)\xi_{i}\}_{g\in G}$ is a frame for $H$. If $A\in \pi(G)'$ such that $\xi_{i} = A\eta_{i}$ and each $\eta_{i}$ is a Bessel vector for $\pi$, then  $A$ is invertible and $\cup_{i=1}^{k}\{\pi(g)\eta_{i}\}_{g\in G}$ is also a frame for $H$. 
\end{lem}

\begin{proof} Let $D$ and $C$ be  the lower frame bound and Bessel bound  for $\cup_{i=1}^{k}\{\pi(g)\xi_{i}\}_{g\in G}$  and  $\cup_{i=1}^{k}\{\pi(g)\eta_{i}\}_{g\in G}$, respectively. 
Then  for every $x\in H$ we 
\begin{eqnarray*}
D||x||^{2} &\le& \sum_{i=1}^{k}\sum_{g\in G}|\langle x, \pi(g)\xi_{i}\rangle |^{2}\\
&=& \sum_{i=1}^{k}\sum_{g\in G}|\langle x, \pi(g)A\eta_{i}\rangle |^{2}\\
&=& \sum_{i=1}^{k}\sum_{g\in G}|\langle A^{*}x, \pi(g)\eta_{i}\rangle |^{2}\\
&\le& C||A^{*}x||^{2},
\end{eqnarray*}
This implies that $A^{*}$ is bounded from below. Since  $A^{*}\in \pi(G)'$ and $\pi(G)'$ is a finite von Neumann algebra, it follows that  $A^{*}$ must be invertible. Thus $A$ is invertible.
\end{proof}

\begin{lem} \label{sublem2.2} \cite{GH, GH1} Let $\pi$ be a projective unitary
representations of a countable group $G$ on a Hilbert space $H$. If $x\in\mathcal{B}_{\pi}$, then there exists a vector $\xi \in M := \overline{span}\{\pi(g)x: g\in G\}$ such that $\{\pi(g)\xi\}_{g\in G}$ is a Parseval frame for $M$. Moreover, $\Theta_{\xi}(H) = \overline{\Theta_{x}(H)}$.
\end{lem}

\begin{lem} \label{lem2.2}  Assume that $\pi$ and $\sigma$ is a commutant pair of projective
representations of a countable group $G$ on a Hilbert space $H$ and $\pi(G)'$ is finite. If $\cup_{i=1}^{k}\{\pi(g)\xi_{i}\}_{g\in G}$ is a frame for $H$, then 
$$
\{\sigma(g)\xi_{1}\oplus ... \oplus \sigma(g)\xi_{k}\}_{g\in G}
$$
is frame sequence in $H^{(k)}$. 

\end{lem}

\begin{proof} Let $$M = \overline{span} \{\sigma(g)\xi_{1}\oplus ... \oplus \sigma(g)\xi_{k}\}_{g\in G}.$$ Then $M$ is $\sigma^{(k)}$-invariant. Note that 
$$
w^{*}(\sigma^{(k)}(G)) = \{A\oplus ... \oplus A: A \in w^{*}(\sigma(G))\}.
$$
So we have that 
$$
w^{*}(\sigma^{(k)}(G)|_{M}) = \{A^{(k)}|_{M}: A \in w^{*}(\sigma(G))\}.
$$
Since $\vec{\xi} =(\xi_{1}, ..., \xi_{k})$ is a Bessel vector for $\sigma^{(k)}$, by Lemma \ref{sublem2.2} we get that there exists a vector $\vec{\eta} = (\eta_{1}, ... , \eta_{k})\in M$ such that
$$
\{\sigma^{(k)}(g)\vec{\eta}\}_{g\in G}
$$
is a Parseval frame for $M$. Now by Lemma \ref{param-lem}  there exists an operator  $T$ in $w^{*}(\sigma^{(k)}(G)|_{M})$ such that 
$T\vec{\eta} = \vec{\xi}$. Write $T = A^{(k)}|_{M}$ for some $A\in w^{*}(\sigma(G))$. Then we  get that $A\eta_{i} = \xi_{i}$ for $i=1, ..., k$ and $A\in \pi(G)'$. 
Thus, by Lemma \ref{lem2.1}, we have that $A$ is invertible, which implies that $T$ is invertible. Hence, by Lemma \ref{param-lem} again, $\{\sigma^{(k)}(g)\vec{\xi}\}_{g\in G}$ is a frame for $M$, which completes the proof.
\end{proof}

We also need the following generalization of Lemma \ref{sublem2.2}. Although it is not a consequence of \ref{sublem2.2}, the proof is very similar and  we  include a sketch proof for reader's convenience.

\begin{lem} \label{lem2.3} {\it Assume that $\pi$ is a projective unitary
representation of a countable group $G$ on a Hilbert space $H$. Suppose that $\cup_{i=1}^{k}\{\pi(g)\xi_{i}\}_{g\in G}$ is a Bessel sequence and let 
$$M = \overline{span} \cup_{i=1}^{k}\{\pi(g)\xi_{i}\}_{g\in G}.$$  Then there exists a vector $\vec{\eta}$ such that

(i) $\cup_{i=1}^{k}\{\pi(g)\eta_{i}\}_{g\in G}$ is a Parseval frame for $M$, and 

(ii) $
\Theta_{\vec{\eta}, \pi}(H) = [\Theta_{\vec{\xi}, \pi}(H)].
$
}
\end{lem}

\begin{proof} It is sufficient to consider the case when $M = H$.  Write $T = \Theta_{\vec{\xi}, \pi}$ and let $T= U|T|$ be its polar decomposition. Then $U$ is an isometry from $H$ into $\ell^{2}(G)^{(k)}$ since the range of $T^*$ is dense in $H$.  It can be easily verified that $T\pi(g) = \lambda^{(k)}(g)T$ for every $g\in G$, where $\lambda$ is the left regular representation for $G$ with the same multiplier as $\pi$. This implies that $U\pi(g) = \lambda^{(k)}(g)U$ for all $g\in G$. Let $\psi_{i} = 0\oplus ... \oplus 0 \oplus \chi_{e}\oplus 0 ... \oplus 0$, where $\chi_{e}$ appears in the $i$-th component, and let $\eta_{i} = U^{*}\psi_{i}$. Then we have
$$
U\pi(g)\eta_{i} = U\pi(g)U^{*}\psi_{i} = UU^{*}\lambda^{(k)}(g)\psi = P\lambda^{(k)}(g)\psi_{i}, 
$$
where $P$ is the orthogonal projection from $\ell^{2}(G)^{(k)}$ onto $[\Theta_{\vec{\xi}, \pi}(H)]$. Since $$\{\lambda^{(k)}(g)\psi_{i}: g\in G, i =1, ... , k\}$$ is an orthonormal basis for $\ell^{2}(G)^{(k)}$, we get that $\{U\pi(g)\eta_{i}: g\in G, i =1, ..., k\}$ is a Parserval frame for $[\Theta_{\vec{\xi}, \pi}(H)]$ and the range space of its analysis operator is $[\Theta_{\vec{\xi}, \pi}(H)]$. Since $U$ is an isometry, we obtain that 
$\cup_{i=1}^{k}\{\pi(g)\eta_{i}\}_{g\in G}$ is a Parseval frame for $H$  and $\Theta_{\vec{\eta}, \pi}(H) = [\Theta_{\vec{\xi}, \pi}(H)]$. 
\end{proof}

Let $\pi$ be a projective unitary representation of $G$ on a Hilbert space $H$  such that $\mathcal{B}_{\pi}$ is dense in $H$. Recall from \cite{DHL-JFA} that two vectors $\xi$ and $\eta$ in $H$ are called {\it $\pi$-orthogonal} if $range(\Theta_{\xi}) \perp range(\Theta_{\eta})$, and {\it $\pi$-weakly equivalent} if $[range(\Theta_{\xi})] = [range(\Theta_{\eta}]$.

The following result obtained in \cite{HL_BLM} characterizes the
$\pi$-orthogonality and $\pi$-weakly equivalence in terms of the
commutant of $\pi(G)$.

\begin{lem}\label{sublem2.4}  Let $\pi$ be a projective representation of a
countable group $G$ on a Hilbert space $H$ such that
$\mathcal{B}_{\pi}$ is dense in $H$. Then two vectors $\xi, \eta\in H$ are

(i) $\pi$-orthogonal if and only if
$[\pi(G)'\xi] \perp [\pi(G)'\eta]$, and 

(ii)  $\pi$-weakly equivalent  if and only if
$[\pi(G)'\xi] = [\pi(G)'\eta]$.
\end{lem}

We need the following (partial) generalization of Lemma \ref{sublem2.4} (ii).

\begin{lem}\label{lem2.4} {\it Let $(\pi, \sigma)$ be a commutant pair of projective  unitary representations of a
countable group $G$ on a Hilbert space $H$ such that $\mathcal{B}_{\pi}$ is dense in $H$. 
 Let $\xi_{i}, \eta_{i}\in H$ $(i =1, .. , k)$ be Bessel vectors for $\pi$. If $[\Theta_{\vec{\xi}, \pi}(H)] = [\Theta_{\vec{\eta}, \pi}(H)]$, then  $[\sigma^{(k)}(G)\vec{\xi}\ ] = [\sigma^{(k)}(G)\vec{\eta}\ ]$.
}
\end{lem}

\begin{proof} By Lemma  \ref{lem-2.1},  we know that $w^{*}(\sigma(G)) = \pi(G)' $ is the closure of the linear span  of $$\{\Theta_{u, \pi}^{*}\Theta_{v, \pi}: u, v \in \mathcal{B}_{\pi}\}$$  in the weak operator topology. Hence $w^{*}(\sigma^{(k)}(G))$ is the (wot)-closure of the linear span  of $$\{ \Theta_{u, \pi}^{*}\Theta_{v, \pi}\oplus ... \oplus \Theta_{u, \pi}^{*}\Theta_{v, \pi}: u, v \in \mathcal{B}_{\pi}\}.$$

 Assume that $\vec{z}= (z_{1}, ... , z_{k})\in [\sigma^{(k)}(G)\vec{\xi}\ ]^{\perp}$. Then for any $u, v\in \mathcal{B}_{\pi}$ we have 
\begin{eqnarray*}
0 &=& \sum_{i=1}^{k}\langle z_{i} , \Theta_{u, \pi}^{*}\Theta_{v, \pi}(\xi_{i})\rangle 
 =  \sum_{i=1}^{k}\langle \Theta_{u, \pi}(z_{i}) , \Theta_{v, \pi}(\xi_{i})\rangle \\
& = & \sum_{i=1}^{k}\langle \Theta_{\xi_{i}, \pi}(v) , \Theta_{z_{i}, \pi}(u)\rangle 
 =  \langle \Theta_{\vec{\xi}, \pi}(v), \Theta_{\vec{z}, \pi}(u)\rangle .
\end{eqnarray*}
This implies $ \Theta_{\vec{z}, \pi}(u) \perp \Theta_{\vec{\xi}, \pi}(v)$. Since $v\in \mathcal{B}_{\pi}$ is arbitray and $\mathcal{B}_{\pi}$ is dense in $H$, we get that 
$ \Theta_{\vec{z}, \pi}(u) \perp [\Theta_{\vec{\xi}, \pi}(H)]$, which implies that $ \Theta_{\vec{z}, \pi}(u) \perp [\Theta_{\vec{\eta}, \pi}(H)]$. Therefore we obtain that
$$
\sum_{i=1}^{k} \langle z_{i} ,  \Theta_{u, \pi}^{*}\Theta_{v, \pi}(\eta_{i}) \rangle =  \langle \Theta_{\vec{\eta}, \pi}(v), \Theta_{\vec{z}, \pi}(u) \rangle  = 0.
$$
This implies that $\vec{z}\in [\sigma^{(k)}(G)\vec{\eta}\ ]^{\perp}$. Hence $[\sigma^{(k)}(G)\vec{\xi}\ ] \subseteq  [\sigma^{(k)}(G)\vec{\eta}\ ]$. Similarly, we also have the reversed inclusion. 
Therefore we obtain $[\sigma^{(k)}(G)\vec{\xi}\ ] = [\sigma^{(k)}(G)\vec{\eta}\ ]$, as claimed. 
\end{proof} 

\begin{lem}\label{lem2.5}  Let $(\pi, \sigma)$ be a  commuting pair of projective unitary representations of a
countable group $G$ on a Hilbert space $H$. If $\{\sigma^{(k)}(g)\vec{\xi}\}_{g\in G}$ is a  Riesz sequence, then $$\overline{span}\cup_{i=1}^{k}\{\pi(g)\xi_{i}\}_{g\in G} = H.$$
 
\end{lem} 

\begin{proof} Assume that $x\perp \pi(g)\xi_{i}$ for all $g\in G$ and $i=1, ... , k$. We need to show that $x = 0$.  Since $w^{*}(\pi(G)) = \sigma(G)'$, we get that $[\sigma(G)'x] \perp [\sigma(G)'\xi_{i}]$. Applying  Lemma \ref{sublem2.4} (i) to $\sigma$, we get that $x$ and $\xi_{i}$ are $\sigma$-orthogonal. This implies that $range(\Theta_{x, \sigma}) \perp range(\Theta_{\xi_{i}, \sigma})$ for every $i$. Let $\vec{x} = (x, 0, ..., 0)\in H^{(k)}$. Then we have  $range(\Theta_{\vec{x}, \sigma^{(k)}}) \perp range( \Theta_{\vec{\xi}, \sigma^{(k)}})$. Since $\{\sigma^{(k)}(g)\vec{\xi}\}_{g\in G}$ is a  Riesz sequence, we know that $ range( \Theta_{\vec{\xi}, \sigma^{(k)}} ) = \ell^{2}(G)$. Thus $range(\Theta_{\vec{x}, \sigma^{(k)}}) = \{0\}$, which implies $\vec{x} = 0$ and hence $x = 0$, as claimed.
\end{proof}

\begin{lem}\label{lem2.6}  Let $(\pi, \sigma)$ be a commutant pair of projective  unitary representations of a
countable group $G$ on a Hilbert space $H$. If $\{\sigma^{(k)}(g)\vec{\xi}\}_{g\in G}$ is a  Riesz sequence, then $\cup_{i=1}^{k}\{\pi(g)\xi_{i}\}_{g\in G}$ is a frame for $H$.
 
\end{lem} 

\begin{proof}  From Lemma \ref{lem2.5} we know that $\overline{span}\cup_{i=1}^{k}\{\pi(g)\xi_{i}\}_{g\in G} = H$. By using  Lemma \ref{lem2.3}, there exists a vector $\vec{\eta}$ such that
$\cup_{i=1}^{k}\{\pi(g)\eta_{i}\}_{g\in G}$ is a Parseval frame for $H$, and  $
\Theta_{\vec{\eta}, \pi}(H) = [\Theta_{\vec{\xi}, \pi}(H)].$ By Lemma \ref{lem2.4}  we get that  $[\sigma^{(k)}(G)\vec{\eta}\ ] = [\sigma^{(k)}(G)\vec{\xi}\ ] $.  Let $M = [\sigma^{(k)}(G)\vec{\xi}\ ]$. Since $\vec{\xi}$ is a frame vector and $\vec{\eta}$ is a Bessel vector for $\sigma^{(k)}|_{M}$, we have by Lemma \ref {param-lem} that there is an operator $T$  in $w^{*}(\sigma^{(k)}(G)|_{M})$ such that $\vec{\eta} = T\vec{\xi}$. Write $ T = (A \oplus ... \oplus A)|_{M}$ with $A\in w^{*}(\sigma(G)) = \pi(G)'$. Then we have
$A\xi_{i} = \eta_{i}$ for $i=1, ... , k$. From Lemma \ref{lem2.1} we get that $A$ is invertible and  $\cup_{i=1}^{k}\{\pi(g)\xi_{i}\}_{g\in G}$ is a  frame for $H$
\end{proof}

\begin{lem}\label{lem2.9} Let $\pi$ be a projective  unitary representations of a
countable group $G$ on a Hilbert space $H$ such that $\mathcal{B}_{\pi}$ is dense in $H$ and it admits a Riesz sequence vector. Suppose that $\xi\in H$ such that $range(\Theta_{\xi, \pi})$ is not dense in $\ell^{2}(G)$. Then there exits a nonzero  vector $x\in H$ such that $range (\Theta_{x, \pi})$ and $range(\Theta_{\xi, \pi})$ are orthogonal.
\end{lem} 
\begin{proof} Let $\psi\in H$ be such that $\{\pi(g)\psi\}_{g\in G}$ is a Riesz sequence and $\Theta_{\psi, \pi} = V |\Theta_{\psi, \pi}|$ be the polar decomposition of its analysis operator. Since $range(\Theta_{\psi, \pi}) = \ell^{2}(G)$, we have that $V$ is a co-isometry.  It can be verified that $V\pi(g) = \lambda(g)V$ and hence we get $\pi(g)V^{*} = V^{*}\lambda(g)$ for every $g\in G$, where $\lambda$ is the left regular projective unitary representation associated with the same multiplier as $\pi$. Now let $P$ be the orthogonal projection onto $[range(\Theta_{\xi, \pi})]$ and $x = V^{*}P^{\perp}\chi_{e}$, where $P^{\perp} = I-P$. Then $P$ commutes with $\lambda$ and so $x \neq 0$.  Now for any $y\in H$ we get 
\begin{eqnarray*}
\Theta_{x, \pi}(y) &=& \sum_{g\in G}\langle y, \pi(g)x\rangle \chi_{g}\\
& = &\sum_{g\in G}\langle y, \pi(g)V^{*}P^{\perp}\chi_{e}\rangle \chi_{g}\\
&=&\sum_{g\in G}\langle y, V^{*}\lambda(g)P^{\perp}\chi_{e}\rangle \chi_{g}\\
&=&\sum_{g\in G}\langle y, V^{*}P^{\perp}\lambda(g)\chi_{e}\rangle \chi_{g}\\
&=&\sum_{g\in G}\langle P^{\perp}Vy, \chi_{g}\rangle \chi_{g}\\
&=&P^{\perp}Vy.
\end{eqnarray*}
Thus $range(\Theta_{x, \pi})$ is contained in the range space of $P^{\perp}$, and therefore we get that  $range (\Theta_{x, \pi})$ and $range(\Theta_{\xi, \pi})$ are orthogonal.
\end{proof}

Now we are ready to prove Theorem \ref{main-thm2}, which follows from the next two theorems. 

\begin{theo}\label{subthm-main3}  Let $\pi$ be a frame representation and  $(\pi, \sigma)$ be 
a dual commutant pair of projective unitary representations of $G$ on $H$ and ${\vec \xi} = (\xi_{1}, ... , \xi_{k})\in H^{(k)}$. Then  $\cup_{i=1}^{k}\{\pi(g)\xi_{i}\}_{g\in G}$ is a frame for $H$
 if and only if $\{\sigma(g)\xi_{1}\oplus ... \oplus \sigma(g)\xi_{k}\}_{g\in G}$ is Riesz sequence.

\end{theo}
\begin{proof} The sufficient part has been proved in Lemma \ref{lem2.6}. To prove the necessary part, let assume that $\cup_{i=1}^{k}\{\pi(g)\xi_{i}\}_{g\in G}$ is a frame for $H$. Then, by Lemma \ref{lem2.2}, we have that $\{\sigma(g)\xi_{1}\oplus ... \oplus \sigma(g)\xi_{k}\}_{g\in G}$ is a frame sequence in $H^{(k)}$. So in order to show that it is a Riesz sequence, it suffices to show that the range space  (which is already closed)  of its analysis operator $\Theta_{\vec{\xi}, \sigma^{(k)}}$ is the entire space $\ell^{2}(G)$. 

Assume to the contrary that $ \Theta_{\vec{\xi}, \sigma^{(k)}}(H^{(k)})\neq \ell^{2}(G)$. By the assumption on dual commutant pairs we know that $\sigma^{(k)}$ admits a Riesz sequence vector, and the set of its Bessel vectors is dense in $H^{(k)}$, we obtain by Lemma \ref{lem2.9} that there exists a nonzero vector $\vec{x}\in H^{(k)}$ such that $range(\Theta_{\vec{x}, \sigma^{(k)}}) \perp range (\Theta_{\vec{\xi}, \sigma^{(k)}})$. 

Let $H_{i} = 0\oplus ... \oplus 0 \oplus H\oplus 0 ... \oplus 0$, where $H$ appears in the $i$-th component. Then we in particular get $\Theta_{\vec{x}, \sigma^{(k)}}(H_{i}) \perp \Theta_{\vec{\xi}, \sigma^{(k)}}(H_{j})$.  Note that $\Theta_{\vec{x}, \sigma^{(k)}}(H_{i})  = \Theta_{x_{i}, \sigma}(H)$ and $\Theta_{\vec{\xi}, \sigma^{(k)}}(H_{j})  = \Theta_{\xi_{j}, \sigma}(H)$. So we have that $\Theta_{x_{i}, \sigma}(H) \perp \Theta_{\xi_{j}, \sigma}(H)$ for all $i, j = 1, ... , k$.  Thus $x_{i}$ and $\xi_{j}$ are $\sigma$-orthogonal for all $i, j = 1, ... , k$. By Lemma \ref{sublem2.4} we get that 
 $[\sigma(G)'x_{i}]\perp [\sigma(G)'\xi_{j}]$ for all $i, j$.  Since $\sigma(G)' = w^{*}(\pi(G))$ we get in particular that for each $i$, $x_{i}\perp [\pi(G)\xi_{j}]$ for $j = 1, ... k$. Hence $x_{i} = 0$ for each $i$ and so $\vec{x} = 0$, which is a contradiction. Therefore we have that  $\Theta_{\vec{\xi}, \sigma^{(k)}}(H^{(k)}) = \ell^{2}(G)$, and hence $\{\sigma^{(k)}(g)\vec{\xi}\}_{g\in G}$ is a Riesz sequence, as claimed.
\end{proof}

\begin{theo}\label{subthm-main2}  Let $\pi$ be a frame representation and  $(\pi, \sigma)$ be
a dual commutant pair of projective unitary representations of $G$ on $H$. Let ${\vec \xi} = (\xi_{1}, ... , \xi_{k})$ . Then we have

(i) $\vec{\xi}$ is a Parserval super-frame vector for $\pi$ if and only if $\{\sigma(g)\xi_{j}: g\in G, j = 1, ..., k\}$ is an orthogonal sequence in $H$ and $||\xi_{i}||^{2} = Tr_{\pi(G)'}(I)$.

(ii) $\vec{\xi}$ is a  super-frame vector for $\pi$ if and only if $\{\sigma(g)\xi_{j}: g\in G, j = 1, ..., k\}$ is Riesz sequence in $H$.

\end{theo}
\begin{proof} 

 (i) First assume that $\vec{\xi}$ is a  complete Parserval super-frame vector for $\pi$. Then each $\xi_{i}$ is a complete Parserval frame vector for $\pi$, and $\xi_{i}$ and $\xi_{j}$ are $\pi$-orthogonal for $i\neq j$. Thus, by Theorem  \ref{thm-main1} and Lemma \ref{sublem2.4}, we get that $\{\sigma(g)\xi_{i}\}_{g\in G}$ is an orthogonal sequence, and $[\pi(G)'\xi_{i}] \perp [\pi(G)'\xi_{j}]$ for $i\neq j$. Therefore we have that $\{\sigma(g)\xi_{j}: g\in G, j = 1, ..., k\}$ is an orthogonal sequence in $H$. The identity follows from Lemma  \ref{lem-2.2} and the fact that each $\xi_{i}$ is a complete Parseval frame vector for $\pi$. Clearly the above argument is reversible, and so we also get the sufficiency part of the proof.

(ii) First assume that $\vec{\xi}$ is super-frame vector for $\pi$. By Lemma \ref{param-lem}, there exists an invertible operator $B =  A\oplus .... \oplus A \in w^{*}(\pi^{(k)}(G))$ such that $B\vec{\xi }$ is a Parserval super-frame vector for $\pi$. This implies by (i) that $\cup_{i=1}^{k}\{\sigma(g)A\xi_{i}\}_{g\in G}$ is an orthogonal sequence. Since $A\in w^{*}(\pi(G)) =  \sigma(G)'$  is invertible, we get that $\sigma(g)\xi_{i} = A^{-1}\sigma A\xi_{i}$, and therefore 
$$\{\sigma(g)\xi_{i}: g\in G, i=1, ... ,k\} = A^{-1} \{\sigma(g)A\xi_{i}: g\in G, i=1, ... ,k\} $$
is a Riesz sequence.

Conversely assume that  $\{\sigma(g)\xi_{j}: g\in G, j = 1, ..., k\}$ is Riesz sequence in $H$. Let $K$ be the closed subspace generated by  $\{\sigma(g)\xi_{j}: g\in G, j = 1, ..., k\}$  and  let $S = \Theta_{\vec{\xi}, \sigma}^{*}\Theta_{\vec{\xi}, \sigma}$.
$$
S = \sum_{i=1}^{k} \Theta_{\xi_{i}, \sigma}^{*}\Theta_{\xi_{i}, \sigma}.
$$
Note that since $ \Theta_{\xi_{i}, \sigma}^{*}\Theta_{\xi_{i}, \sigma}\in \sigma(G)' = w^{*}(\pi(G))$, we obtain that $S\in w^{*}(\pi(G))$.
Moreover, $S|_{K}: K\rightarrow K $ is positive invertible.  Write $T = (S|_{K})^{-1/2}$. Then $T$ commutes with $\sigma(g)$ when restricted to $K$ for all $g\in G$. Thus we obtain that 
$$
\cup_{i=1}^{k}\{\sigma(g)T\xi_{i}: g\in G\} = T\cup_{i=1}^{k}\{\sigma(g)\xi_{i}: g\in G\}
$$
is an Parseval frame for $K$. Since $T$ is invertible and  $\{\sigma(g)\xi_{j}: g\in G, j = 1, ..., k\}$ is Riesz basis for $K$, we get that $\cup_{i=1}^{k}\{\sigma(g)T\xi_{i}: g\in G\}$ is an orthogonal basis for $K$. Select $c_{i} > 0$ such that $c_{i}^{2}||T\xi_{i}||^{2} = Tr_{\pi(G)'}(I)$ and write $\eta_{i} = c_{i}T\xi_{i}$. Then  $\cup_{i=1}^{k}\{\sigma(g)\eta_{i}: g\in G\}$ is an orthogonal sequence with $||\eta_{i}||^{2} = Tr_{\pi(G)'}(I)$. Thus, by (i) we get that 
$\eta = (\eta_{1}, ... , \eta_{k})$ is a complete Parserval super-frame vector for $\pi$. This implies that  $\vec{\phi} := (T\xi_{1}, .... , T\xi_{k})$ is also a complete super-frame vector for $\pi$.

Let $P$ be the orthogonal projection from $H$ onto $K$. Then $P\in w^{*}(\pi(G))$ since $K$ is invariant under $\sigma(G)$.  Define $A = T \oplus P^{\perp}$ and 
$$
B =  A\oplus .... \oplus A.
$$
Then $B\in w^{*}(\pi^{(k)}(G))$ is invertible  and $B\vec{\xi} = \vec{ \phi}$. This implies by Lemma \ref{param-lem} that $\vec{\xi} = B^{-1}\vec{\phi}$ is a complete frame vector for $\pi^{(k)}$, i.e., $\xi$ is super-frame vector for $\pi$. 
\end{proof} 


\noindent{\bf A Final Remark:} The mains results of this paper have been presented by Deguang Han at several conferences including  the 2015  workshop on ``Aperiodic Order and Signal Analysis" at  Norwegian University of Science and Technology and  the FFT talk at the University of Maryland in 2016. In a recent preprint \cite{Franz} Jakobsen and Luef  are also able to get Corollary \ref{coro-1} by using a complete different approach. The purposes of the two approaches are also quite different. This paper is focused on establishing the general duality principle for arbitrary groups and on building its connections with the theory of operator algebras and group representations. In this context the duality principle for Gabor representation is only a special case of a more general duality phenomenon for arbitrary projective unitary representations. Paper \cite{Franz} is focused on examining the Gabor theory from the perspective of projective modular theory and it  (at least in the discrete setting) only applies to Gabor representations.




%
%
%

\end{document}